\documentclass[reqno,11pt]{amsart}
\usepackage{latexsym}
\usepackage{amsfonts}
\usepackage{amssymb}
\usepackage{mathrsfs}
\usepackage[all]{xy}
\usepackage{bbm}
\usepackage{color}

\newtheorem{theorem}{Theorem}[section]

\newtheorem{lemma}[theorem]{Lemma}
\newtheorem{conj}[theorem]{Conjecture}
\newtheorem{proposition}[theorem]{Proposition}
\theoremstyle{definition}

\newtheorem{remark}[theorem]{Remark}
\numberwithin{equation}{section}

\newcommand{\RR}{\mathbb{R}}
\newcommand{\NN}{\mathbb{N}}
\newcommand{\QQ}{\mathbb{Q}}
\newcommand{\ZZ}{\mathbb{Z}}

\newcommand{\cV}{\mathcal{V}}

\newcommand{\cL}{\mathcal{L}}
\newcommand{\cH}{\mathcal{H}}

\newcommand{\cN}{\mathcal{N}}

\newcommand{\sA}{\mathscr{A}}

\newcommand{\sC}{\mathscr{C}}
\newcommand{\sB}{\mathscr{B}}
\newcommand{\Bad}{\mathbf{Bad}}
\newcommand{\bzero}{\mathbf{0}}
\newcommand{\ba}{\mathbf{a}}
\newcommand{\bc}{\mathbf{c}}
\newcommand{\bx}{\mathbf{x}}

\newcommand{\bz}{\mathbf{z}}

\newcommand{\bv}{\mathbf{v}}
\newcommand{\bp}{\mathbf{p}}

\newcommand{\br}{\mathbf{r}}
\newcommand{\bw}{\mathbf{w}}

\newcommand{\diag}{\mathrm{diag}}

\newcommand{\Vol}{\mathrm{Vol}}
\newcommand{\SL}{\mathrm{SL}}
\newcommand{\GL}{\mathrm{GL}}

\newcommand {\ignore}[1]  {}

\newif\ifdraft\drafttrue

\draftfalse

\newcommand{\ggm}{G/\Gamma}

\newcommand{\rr}{{\bf r}}
\renewcommand{\emptyset}{\varnothing}
\renewcommand{\setminus}{\smallsetminus}

\begin{document}

\title{Bounded orbits of certain diagonalizable flows on $SL_{n}(\RR)/SL_{n}(\ZZ)$}

\author{Lifan Guan}
\address{Beijing International Center for Mathematical Research, Peking University, Beijing 100871, China.}
\email{guanlifan@pku.edu.cn}

\author{Weisheng Wu}
\address{School of Mathematical Sciences, Peking University, Beijing 100871, China.}
\email{wuweisheng@math.pku.edu.cn}

\begin{abstract}
We prove that the set of points that have bounded orbits under certain diagonalizable flows is a hyperplane absolute winning subset of $SL_{n}(\RR)/SL_{n}(\ZZ)$.
\end{abstract}

\maketitle

\section{Introduction}
\subsection{Statement of main results}
Let $G$ be a connected Lie group, $\Gamma$ a nonuniform lattice in $G$, and $F=\{g_t: t \in \RR \}$ a one-parameter subgroup of $G$ with noncompact closure. We are interested in the dynamical properties of the action of $F$ on the homogeneous space $\ggm$ by left translations. Specifically, we will focus on the study of
the set
\begin{equation*}
E(F) := \{\Lambda\in G/\Gamma: F\Lambda\text{ is bounded in $\ggm$}\}
\end{equation*}
in this paper. In certain important cases, it turns out that $E(F)$ has zero Haar measure (For example when $G$ is semisimple without compact factors and $\Gamma$ is irreducible, this follows from Moore's ergodicity theorem). If $F$ is Ad-unipotent, $E(F)$ is even smaller. In this case, by Ratner's Theorems, $E(F)$ is contained in a countable union of proper submanifolds, hence has Hausdorff dimension $< \dim G$. When $F$ is Ad-semisimple, the
situation is quite different. Motivated by the work of Dani (cf. \cite{D1}, \cite{D2}), Margulis proposed a conjecture in his 1990's ICM report \cite{Mar}, which was settled in a subsequent work of Kleinbock and Margulis \cite{KM}. In that work, they proved: if the flow $(\ggm, F)$ has the so-called property (Q), then the set $E(F)$
is \emph{thick}, i.e. for any nonempty open subset $V$ of $\ggm$ the set $E(F) \cap V$ is of Hausdorff dimension equal to the dimension of the underlying space $\ggm$. In particular, when $F$ is Ad-semisimple, the flow $(\ggm, F)$ always has property (Q).

Given countably many Ad-semisimple $F_n$, it is natural to ask whether the set of points $\Lambda$ such that all the orbits $F_n\Lambda$ are bounded is still thick. This is natural from both the dynamical point of view and its relation to number theory. This is proved to be true for $G = \SL_2(\RR)$ and $\Gamma=\SL_2(\ZZ)$ in \cite{KW3}, and for $G=\SL_3(\RR)$ and $\Gamma=\SL_3(\ZZ)$ in \cite{AGK}. Note that this set is the intersection $\bigcap_nE(F_n)$. A powerful tool for studying intersection properties of different sets is a type of game introduced by Schmidt in \cite{S}, which is called  Schmidt's  $(\alpha, \beta)$-game. The game can be played on any metric space, and it defines a class of so-called $\alpha$-winning sets ($0<\alpha<1$). When the metric space is a Riemannian manifold, $\alpha$-winning sets are thick and stable with respect to countable intersections. In this paper, we will use a variant of Schmidt's  $(\alpha, \beta)$-game, i.e., the hyperplane absolute game introduced in \cite{BFKRW} and \cite{KW3}. This game has the advantage that it can be naturally defined on a differential manifold without picking a Riemannian metric while the hyperplane absolute winning (abbreviated as HAW) sets also enjoy the thickness and countable intersection properties. See Section $2$ for details. Note that, in both \cite{KW3} and \cite{AGK}, the authors prove their results by showing that $E(F)$ is HAW in the corresponding case. In fact, the following conjecture is proposed in \cite{AGK}.
\begin{conj}\cite[Conjecture 7.1]{AGK}
Let $G$ be a Lie group, $\Gamma$ a lattice in $G$, and $F$ a one-parameter Ad-diagonalizable subgroup of $G$. Then the set $E(F)$ is HAW on $\ggm$.
\end{conj}

In this paper, we restrict ourselves to the case
\begin{equation*}
G=SL_{n}(\RR), \quad \Gamma=SL_{n}(\ZZ).
\end{equation*}
Our main theorem is the following, verifying the above conjecture for certain class of $F$.
\begin{theorem}\label{main1}
Let $F$ be a one-parameter subgroup of $G$ satisfying the following property,
\begin{equation}\label{e:tame}
\begin{aligned}
\text{it is diagonalizable and the eigenvalues of\ } &g_1 (\text{\ denoted by\ } \lambda_1,\ldots,\lambda_n) \text{\ satisfy}:\\
\#\{i:|\lambda_i|<1\}=1 \text{ and } \#\{i:|\lambda_i|&=\max_{1\leq j\leq n}|\lambda_j|\}\geq n-2.
\end{aligned}
\end{equation}
 Then the set $E(F)$ is HAW on $\ggm$.
\end{theorem}
 We also prove the following theorem verifying \cite[Conjecture 7.2]{AGK} for $F$ satisfying \eqref{e:tame}.
\begin{theorem}\label{main2}
Let $F$ be a one-parameter subgroup of $G$ satisfying \eqref{e:tame}, and $F^+=\{g_t\in F:t\geq 0\}$. Let $H(F^+)$ denote the expanding horospherical subgroup of $F^+$ which is defined as
\begin{equation}\label{def-horo}
H(F^+)=\left\{ h \in G: \lim_{t \to +\infty} g_t^{-1} h g_t=e \right\}.
\end{equation}
Then for any $\Lambda\in \ggm$, the set
\begin{equation*}
\{h\in H(F^+) : h\Lambda \in E(F^+)\}
\end{equation*}
is HAW on $H(F^+) $.
\end{theorem}

\subsection{Connection to number theory}

To begin, let us define a \emph{$d$-weight $\rr$} to be a $d$-tuple $\rr=(r_1,\ldots,r_d)\in \RR^d$ such that each $r_i$ is positive and their sum equals $1$.
Due to work of Dani \cite{D1} and Kleinbock \cite{Kl}, we know that for a $d$-weight $\rr$ there is a close relation between the set of $\rr$-badly approximable vectors (abbreviated as $\Bad(\rr)$)  and bounded orbits of certain flow corresponding to $\rr$ in $\SL_{d+1}(\RR)/\SL_{d+1}(\ZZ)$. We will not present the explicit definition of $\Bad(\rr)$ here. But we remark that they are natural generalizations of the classical badly approximable numbers.   Recently, there is a rapid progress on the study of intersection properties of the sets $\Bad(\rr)$ for different weight $\rr$, for example, see \cite{BPV,An2,An,Ber,NS,GY}. Concerning the winning properties of such sets, Schmidt proved that $\Bad_d$ (abbreviation for $\Bad(\frac{1}{d},\ldots,\frac{1}{d})$) is winning for his game for any $d\in \mathbb{N}$. They are also proved to be HAW in \cite{BFKRW}. Recently, An \cite{An} proved that $\Bad(\rr)$ are winning sets for Schmidt's game for any $2$-weight $\rr$. The HAW property is also established for such sets by Nesharim and Simmons \cite{NS}.
To this end, we want to highlight the following theorem proved in \cite{GY}, since it motivates the results of this paper.
\begin{theorem}(Cf. \cite[Theorem 1.4]{GY})\label{GY}
Let a $d$-weight $\rr=(r_1,\ldots,r_d)$ satisfy
\begin{equation}\label{def-r-special}
\sum_{i=1}^{d} r_i=1 \text{ and } r_1=\ldots=r_{d-1}\ge r_d\ge 0.
\end{equation}
Then $\Bad(\rr)$ is HAW.
\end{theorem}


\begin{remark}
Whether $\Bad(\rr)$ is winning ($\alpha$-winning or HAW) for general weight $\rr$ is a challenging open problem proposed by Kleinbock \cite{Kl}.
\end{remark}

\subsection{Structure of the paper}
For sake of convenience, from now on we will assume
\begin{equation*}
G=\SL_{d+1}(\RR), \Gamma=\SL_{d+1}(\RR).
\end{equation*}
That is, the number $n$ in the title of the paper is replaced by $d+1$.

The paper is organized as follows. In Section 2 we recall some basics of certain Schmidt games, namely the hyperplane absolute game and the hyperplane potential game. In Section \ref{Dio-char}, we state Theorem \ref{main3} and then convert it to the Diophantine setting using Lemma \ref{Dani-Klein}. Note that Theorem \ref{main3} whose proof forms the most technical part of this paper, can be regarded as a special case of Theorem \ref{main2}.
 In the rest of Section 3, we turn to the study of pairs $(B,P)$, where $B$ is a closed ball in $\RR^{2d-1}$ and $P$ is a rational vector in $\QQ^d$. We manage to attach a rational hyperplane and a rational line in $\RR^d$ to the pair $(B,P)$. Section 5 is the core of this paper, in which Theorem \ref{main3} is proved using the information of the pairs $(B,P)$ and some subdivisions prepared in Section 3 and 4. In the last section, Theorem \ref{main1} and Theorem \ref{main2} are deduced from Theorem \ref{main3}.

\section{Schmidt games}

In this section, we will recall definitions of certain Schmidt games, namely, the hyperplane absolute game and the hyperplane potential game. They are both variants of the $(\alpha,\beta)$-game introduced by Schmidt in \cite{S}. Since we don't make a direct use of the $(\alpha,\beta)$-game in this paper, we omit its definition here and refer the interested reader to \cite{S,S2}. Instead, we list here some nice properties of the $\alpha$-winning sets:
\begin{enumerate}
\item If the game is played on a Riemannian manifold, then any $\alpha$-winning set is thick.
\item The intersection of countably many $\alpha$-winning sets is $\alpha$-winning.
\end{enumerate}

\subsection{Hyperplane absolute game}
The hyperplane absolute game was introduced in \cite{BFKRW}. It is played on an Euclidean space $\RR^{d}$.
Given a hyperplane $L$ and a $\delta>0$, we denote by $L^{(\delta)}$ the $\delta$-neighborhood of $L$, i.e.,
$$L^{(\delta)}:=\{\bx\in\RR^{d}:\mathrm{dist}(\bx,L)<\delta\}.$$ For $\beta\in(0,\frac{1}{3})$, the
\emph{$\beta$-hyperplane absolute game} is defined as follows. Bob starts by choosing a
closed ball $B_0\subset \RR^d$ of radius $\rho_0$. In the $i$-th turn, Bob chooses a closed ball $B_i$ with radius $\rho_i$, and then
Alice chooses a  hyperplane neighborhood $L_i^{(\delta_i)}$ with $\delta_i\le\beta \rho_i$. Then in the $(i+1)$-th turn, Bob chooses a closed ball $B_{i+1}\subset B_i\setminus L_i^{(\delta_i)}$ of radius $\rho_{i+1}
\geq\beta \rho_{i}$. By this process there is a nested sequence of closed balls
$$B_0\supseteq B_1\supseteq B_2\supseteq \ldots.$$
We say that a subset $S\subset \RR^{d}$ is \emph{$\beta$-hyperplane absolute
winning} (\emph{$\beta$-HAW} for short) if no matter how Bob plays, Alice can ensure that
$$\bigcap_{i=0}^\infty B_i\cap S\neq \emptyset.$$ We say $S$ is \emph{hyperplane absolute winning}
(\emph{HAW} for short) if it is $\beta$-HAW for any $\beta\in(0,\frac{1}{3})$.

We have the following lemma collecting the basic properties of $\beta$-HAW subsets and HAW subsets of $\RR^d$ (\cite{BFKRW}, \cite{KW3}, \cite{GY}):
\begin{lemma}\label{haw-property}
\begin{enumerate}
\item A HAW subset is always $\frac{1}{2}$-winning.
\item Given $\beta,\beta'\in(0,\frac{1}{3})$, if $\beta\geq\beta'$, then any $\beta'$-HAW set is $\beta$-HAW.
\item A countable intersection of HAW sets is again HAW.
\item Let $\varphi:\RR^{d}\to \RR^{d}$ be a $C^1$ diffeomorphism. If $S$ is a HAW set, then so is
$\varphi(S)$.
\end{enumerate}
\end{lemma}

The notion of HAW was extended to subsets of $C^1$ manifolds in \cite{KW3}. This is done in two steps. First, one defines the hyperplane absolute game on an open subset $W\subset \RR^d$. It is defined just as the hyperplane absolute game on $\RR^d$, except for requiring that Bob's first move $B_0$ be contained in $W$. Now, let $M$ be a $d$-dimensional $C^1$ manifold, and let $\{(U_{\alpha},\phi_{\alpha})\}$ be a $C^1$ atlas on $M$. A subset $S\subset M$ is said to be HAW on $M$ if for each $\alpha$, $\phi_{\alpha}(S\cap U_{\alpha})$ is HAW on $\phi_{\alpha}(U_{\alpha})$. The definition is independent of the choice of atlas by the property (4) listed above. We have the following lemma that collects the basic properties of HAW subsets of a $C^1$ manifold (cf. \cite{KW3}).

\begin{lemma}\label{haw-manifold}

\begin{enumerate}
\item HAW subsets of a $C^1$ manifold are thick.
\item A countable intersection of HAW subsets of a $C^1$ manifold is again HAW.
\item Let $\phi:M\rightarrow N$ be a diffeomorphism between $C^1$ manifolds, and let $S\subset M$ be a HAW subset of $M$. Then $\phi(S)$ is a HAW subset of $N$.
\item Let $M$ be a $C^1$ manifold with an open cover $\{U_{\alpha}\}$. Then, a subset $S\subset M$ is HAW on $M$ if and only if $S\cap U_{\alpha}$ is HAW on $U_{\alpha}$ for each $\alpha$.
\item Let $M,N$ be $C^1$ manifolds, and let $S\subset M$ be a HAW subset of $M$. Then $S\times N$ is a HAW subset of $M\times N$.
\end{enumerate}
\end{lemma}

\subsection{Hyperplane potential game}

Being introduced in \cite{FSU}, the hyperplane potential game also defines a class of subsets of $\RR^{d}$
called \emph{hyperplane potential winning} (\emph{HPW} for short) sets. The following lemma allows one to prove
the HAW property of a set $S\subset \RR^d$ by showing that it is winning for the hyperplane potential game. And this is exactly the game we will use in this paper.
\begin{lemma}(cf. \cite[Theorem C.8]{FSU})\label{HPW}
A subset $S$ of $\RR^{d}$ is HPW if and only if it is HAW.
\end{lemma}

The hyperplane potential game involves two parameters $\beta\in(0,1)$ and $\gamma>0$.  Bob starts the game
 by choosing a closed ball $B_0\subset \RR^{d}$ of radius $\rho_0$. In the $i$-th turn,
Bob chooses a closed ball $B_i$ of radius $\rho_i$, and then Alice chooses a countable family of hyperplane
neighborhoods $\{L_{i,k}^{(\delta_{i,k})}: k\in \NN\}$ such that
\begin{equation*}
\sum_{k=1}^\infty \delta_{i,k}^\gamma\le(\beta \rho_{i})^\gamma.
\end{equation*}
Then in the $(i+1)$-th turn, Bob chooses a closed ball $B_{i+1}\subset B_i$ of radius $\rho_{i+1}\ge\beta \rho_{i}$. By
this process there is a nested sequence of closed balls
$$B_0\supseteq B_1\supseteq B_2\supseteq \ldots.$$
We say a subset $S\subset \RR^{d}$ is \emph{$(\beta,\gamma)$-hyperplane potential winning}
(\emph{$(\beta,\gamma)$-HPW} for short) if no matter how Bob plays, Alice can ensure that
$$\bigcap_{i=0}^\infty B_i\cap\Big(S\cup\bigcup_{i=0}^\infty\bigcup_{k=1}^\infty
L_{i,k}^{(\delta_{i,k})}\Big)\ne\emptyset.$$ We say $S$ is \emph{hyperplane potential winning}
(\emph{HPW} for short) if it is $(\beta,\gamma)$-HPW for any $\beta\in(0,1)$ and $\gamma>0$.

\section{Converting to the Diophantine setting}
Fix $d\ge 2$. Recall that we have assumed
\begin{equation*}
G=\SL_{d+1}(\RR), \Gamma=\SL_{d+1}(\RR)
\end{equation*}
to simplify notations. That is, the number $n$ in the title of the paper is replaced by $d+1$. Let
\begin{equation*}
\pi:G\to \ggm \text{ be the natural projection.}
\end{equation*}
We will fix a $d$-weight $\rr$ satisfying \eqref{def-r-special} until the last section. For simplicity, sometimes we also write $\lambda=r_1=\cdots=r_{d-1}$, $\mu=r_d$. Both Theorem \ref{main1} and Theorem \ref{main2} will be deduced from the following theorem.
\begin{theorem}\label{main3}
Let $\rr$ be a weight satisfying \eqref{e:tame}. Denote
\begin{equation*}
F_{\br} := \{g_t=\diag(e^{r_1t},e^{r_2t}, \cdots, e^{r_dt}, e^{-t}): t\in \RR\}, \quad  F^+_{\br}:=\{g_t\in F_{\br}:t\geq 0 \},
\end{equation*}
and
\begin{equation}\label{def-u}
U:=\left\{
u_{\bx,y,\bz}: \bx,\bz\in \RR^{d-1}, y\in \RR
\right\}, \quad \text{where } u_{\bx,y,\bz}:=\begin{pmatrix}
  Id&\bz&\bx \\&1&y\\&&1
\end{pmatrix} \in G.
\end{equation}
 Then the set $U\cap \pi^{-1}(E(F^+_\br))$ is HAW on $U$.
\end{theorem}
\begin{remark}
 If $\rr$ satisfies $r_1>r_d$ in advance, then the expanding horospherical subgroup $H(F^+_{\br})$ defined as in \eqref{def-horo} coincides with the group $U$ given in \eqref{def-u}. Thus in this case, Theorem \ref{main3} can be regarded as special case of the Theorem \ref{main2} with $\Lambda=\Gamma$.
\end{remark}

\subsection{Diophantine characterization}\label{Dio-char}
For technical reasons, we will prove Theorem \ref{main3} by applying the diffeomorphism $\RR^{2d-1} \to U$ defined as
$$(\bx,y,\bz) \mapsto u_{\bx,y,\bz}^{-1}.$$

\begin{remark}\label{diffeo}
In view of Lemma \ref{haw-manifold}(3), if we can prove the set $$\{(\bx,y,\bz)\in \RR^{2d-1}: F^+_{\br}u_{\bx,y,\bz}^{-1}\Gamma \text{ is bounded in } \ggm \}$$ in HAW on $\RR^{2d-1}$, then Theorem \ref{main3} will follows.
\end{remark}

A rational vector $P\in \QQ^d$ will be always written in the following reduced form:
\begin{equation*}
\text{$P=(\frac{\bp}{q},\frac{r}{q})$, with $q>0$ and $\bp=(p_1,\dots,p_{d-1})$ satifying $gcd(p_1,\dots,p_{d-1},r,q)=1$.}
\end{equation*}
Such a form is unique, thus we may write the denominator of $P$ as a function $q(P)$.

We need the following Diophantine characterization of the boundedness of $F^+_{\br}u_{\bx,y,\bz}^{-1}\Gamma$ in $\ggm$. For $\epsilon>0$ and a rational vector $P=(\frac{\bp}{q},\frac{r}{q}) \in \QQ^d$ written in its reduced form, we denote
\begin{equation*}
\begin{aligned}
\Delta_\epsilon(P):=&\Big\{(\bx,y,\bz)\in\RR^{d-1}\times\RR\times\RR^{d-1}:\Big|y-\frac{r}{q}\Big|<\frac{\epsilon}{q^{1+\mu}},\\
&\left\|\bx-\frac{\bp}{q}-\Big(y-\frac{r}{q}\Big)\bz\right\|_{\infty}<\frac{\epsilon}{q^{1+\lambda}}\Big\},
\end{aligned}
\end{equation*}
where $\|\cdot\|_{\infty}$ means the maximal norm on $\RR^{d-1}$, that is, for $\bx=(x_1,\ldots,x_{d-1})$, $\|\bx\|_{\infty}=\max\{|x_1|,\ldots,|x_d|\}$.
Then we set
\begin{equation*}
S_{\epsilon}(\rr):=\RR^{2d-1}\setminus\bigcup_{P\in\QQ^d}\Delta_{\epsilon}(P)
\end{equation*}
and
\begin{equation*}
S(\rr):=\bigcup_{\epsilon>0}S_{\epsilon}(\rr).
\end{equation*}

The following lemma allows us to convert our problem to the Diophantine setting. For the proof, one can refer to \cite{Kl} (see also \cite[Lemma 3.2]{AGK}).
\begin{lemma}\label{Dani-Klein}(cf. \cite[Theorem 2.5]{Kl})
\label{L:diophantine}
The orbit $F^+_{\br}u_{\bx,y,\bz}^{-1}\Gamma$ is bounded if and only if $(\bx,y,\bz)\in S(\rr)$, that is, there is $\epsilon=\epsilon(\bx,y,\bz)>0$ such that
\begin{equation*}
\max \left\{q^\mu|qy-r|, q^\lambda\|q\bx-\bp-(qy-r)\bz\|_{\infty}\right\}\ge \epsilon, \ \forall P=(\frac{\bp}{q},\frac{r}{q}) \in \QQ^d.
\end{equation*}
\end{lemma}


\subsection{Attaching hyperplanes}\label{subsc-hyperplane}

Let $\sB$ denote the set of closed balls in $\RR^{2d-1}$ with radius smaller than $1/d$. We shall introduce a function
\begin{equation*}
\ba^+: \sB\times \QQ^d \to \ZZ^d
\end{equation*}
below which enables us to define a linear function on $\RR^d$ that depends on the pair of a closed ball $B\in\sB$ and $ P=(\frac{\bp}{q},\frac{r}{q}) \in \QQ^d$:
\begin{equation}\label{def-f-bp}
F_{B,P}(\bw)=\ba^+(B,P)\cdot\bw-\ba^+(B,P)\cdot\left(\frac{\bp}{q},\frac{r}{q}\right), \quad \bw\in \RR^d.
\end{equation}
We also write for simplicity
\begin{equation}\label{def-c-bp}
 C(B,P)=\ba^+(B,P)\cdot\left(\frac{\bp}{q},\frac{r}{q}\right).
 \end{equation}
Finally we can define a hyperplane attached to the pair $(B,P)$ to be
\begin{equation*}
\cH_{B,P}:=\{\bw\in \RR^d: F_{B,P}(\bw)=0\}.
\end{equation*}

Now let us define the function $\ba^+: \sB\times \QQ^d \to \ZZ^d$. We shall need the following lemma:
\begin{lemma}\label{L:BPV}
Let $\bz\in\RR^{d-1}$. For any $P=(\frac{\bp}{q},\frac{r}{q}) \in \QQ^d$, there exists $(\ba,b)\in\ZZ^{d}$ with $(\ba,b)\ne(\bzero,0)$ such that $\ba\cdot\bp+br\in q\ZZ$ and $\|\ba\|_{\infty}\le q^\lambda$, $|b+\bz\cdot\ba|\le q^\mu$.
\end{lemma}

\begin{proof}
By Minkowski's linear forms theorem (cf. \cite[Chapter III, Theorem III]{Ca}), there exist $\ba\in \ZZ^{d-1},b,c\in\ZZ$ which are not all zero, such that
$$|\ba\cdot\bp+br+cq|<1,\quad \|\ba\|_{\infty}\le q^\lambda, \quad |b+\bz\cdot\ba|\le q^\mu.$$
Since $\ba\cdot\bp+br+cq\in\ZZ$, it must be $0$ by the first inequality above. Assume that $\ba=\bzero$ and $b=0$. Then it follows from $\ba\cdot\bp+br+cq=0$ and $q\ne0$ that $c=0$, which is a contradiction. Thus $(\ba,b)\ne(\bzero,0)$. The lemma follows.
\end{proof}

Now let us consider the following set
\begin{align*}
\sA_{B,P}:=\big\{(\ba,b)\in\ZZ^{d}: \ &  (\ba,b)\ne(\bzero,0), \ba\cdot\bp+br\in q\ZZ, \\
& \|\ba\|_{\infty}\le q^\lambda, |b+\bz_B\cdot \ba|\le q^\mu+\rho(B)^{\frac{1}{2}} \big\},
\end{align*}
where $\bz_B$ is the $\bz$-coordinate of the center of $B$ and $\rho(B)$ is the radius of $B$. It follows from Lemma \ref{L:BPV} that $\sA_{B,P}\ne\emptyset$. We choose and fix
$$\ba^+(B,P)=\big(\ba(B,P),b(B,P)\big)\in\sA_{B,P}$$
such that
\begin{align}
\xi(B,P):=&\ \max\Big\{\|\ba(B,P)\|_{\infty}, |b(B,P)+\bz_B \cdot\ba(B,P)|\Big\} \notag\\
=&\ \min\Big\{\max\big\{\|\ba\|_{\infty}, |b+\bz_B\cdot \ba|\big\}:(\ba,b)\in\sA_{B,P}\Big\}. \label{E:max}
\end{align}
This completes the definition of the function $\ba^+$.
Then we define the \emph{height of $P$ with respect to $B$}:
$$H_B(P):=q(P)\xi(B,P).$$

\begin{remark}
From its definition, one can see that the height function $H_B(P)$ is not canonically defined, i.e. it may depend on a choice. But we have the following lemma  controlling the size of $H_B(P)$.
\end{remark}

\begin{lemma}\label{L:q}
For any $(B,P)\in \sB\times \QQ^d$, we have
\begin{equation}\label{E:qq}
q(P)\le H_B(P)\le q(P)^{1+\lambda}.
\end{equation}
\end{lemma}

\begin{proof}
Write $q(P)$ simply as $q$, the first inequality is clear from the definition. By Lemma \ref{L:BPV}, $\sA_{B,P}$ contains a vector $(\ba_0,b_0)$ with $\|\ba_0\|_{\infty}\le q^\lambda$ and $|b_0+\bz_B \cdot\ba_0|\le q^\mu$. Thus, it follows from \eqref{E:max} that
\begin{equation*}
\begin{aligned}
\max\Big\{\|\ba(B,P)\|_{\infty}, |b(B,P)+\bz_B \cdot \ba(B,P)|\Big\}&\le\max\Big\{\|\ba_0\|_{\infty} ,|b_0+\bz_B \cdot \ba_0|\Big\} \\
&\le\max\{q^\lambda,q^\mu\}=q^\lambda.
\end{aligned}
\end{equation*}
The second inequality follows.
\end{proof}

\begin{remark}
It follows from the definition of $\ba^+(B,P)$ that $C(B,P)\in \ZZ$, thus the coefficients of $F_{B,P}$ belong to $\ZZ$.
\end{remark}

\subsection{Attaching lines}
We shall define another function
\begin{equation*}
\bv^+: \sB\times \QQ^d \to \QQ^d
\end{equation*}
in this subsection. The function $\bv^+(*,P)$ takes values in the lattice $\Lambda_P$ which is defined as follows:
$$\Lambda_P=\ZZ^d+\ZZ\left(\frac{\bp}{q},\frac{r}{q}\right),\quad \text{where } P=\left(\frac{\bp}{q},\frac{r}{q}\right).$$
The line attached to the pair $(B,P)$ is defined to be
\begin{equation*}
\cL_{B,P}:=\left\{\bw\in \RR^d: \bw-\left(\frac{\bp}{q},\frac{r}{q}\right)=t \bv^+(B,P), \quad t\in \RR\right\}.
\end{equation*}

The definition of the function $\bv^+$ is given in the following lemma.
\begin{lemma}
For any $(B,P)\in \sB\times \QQ^d$, there exists a non-zero vector
$$\bv^+(B,P)=(\bv(B,P),u(B,P))\in \Lambda_P$$
with $\bv(B,P)\in \RR^{d-1}, u(B,P)\in \RR$ such that
\begin{equation}\label{def-parallel}
\|\bv(B,P)-u(B,P)\bz_B\|_{\infty}\le 2dq(P)^{-\lambda}, \quad |u(B,P)|\le 2d\xi(B,P)q(P)^{-\lambda-\mu}.
\end{equation}
\end{lemma}

\begin{proof}
Write $q(P)$ simply as $q$. It is easy to check that $d(\Lambda_P)=1/q$, where $d(\Lambda_P)$ denotes the covolume of the lattice $\Lambda_P$. We will make use of the vector $\ba^+(B,P)$ constructed in the previous subsection. For simplicity, write $\ba^+(B,P),\ba(B,P),a_i(B,P),b(B,P),\xi(B,P),\bz_B$ as $\ba^+,\ba,a_i,b,\xi,\bz$ respectively.

We have the following two distinct cases:
\begin{enumerate}
  \item \textbf{Case  $|a_k|q^{-\lambda}=\max(|a_1|q^{-\lambda},\cdots,|a_{d-1}|q^{-\lambda},|b+\bz\cdot \ba|\xi q^{-\lambda-\mu})$, where $1\le k\le d-1$.}\\
      Then it is obvious that $a_k\ne 0$. 
      Consider the convex body
      \begin{align*}
      \Sigma_k:=\big\{\bw=(w_1,\ldots,w_d)\in \RR^d: \ &|w_i-z_iw_d|\le q^{-\lambda}, i\ne k,d;\\
       &|w_d|\le \xi q^{-\lambda-\mu}; \quad |\ba^+\cdot\bw|<1 \big\}.
      \end{align*}
      A direct computation shows that
      \begin{equation*}
      2^{-d}\Vol(\Sigma_k)=|a_k|^{-1}\xi q^{-\lambda-\mu}\prod_{i\ne k,d}q^{-\lambda}=|a_k|^{-1}\xi q^{-1}\ge q^{-1}.
      \end{equation*}
  Hence there is a non-zero $\Lambda_P$-lattice point $\tilde{\bw}=(\tilde{w}_1,\ldots,\tilde{w}_d)$ in $\Sigma_k$. Moreover, since $|\ba^+\cdot\tilde{\bw}|<1$ implies $\ba^+\cdot\tilde{\bw}=0$, we have
  \begin{align*}
  |\tilde{w}_k-z_k\tilde{w}_d|&=|a_k|^{-1}\Big|\sum_{i\ne k,d}a_i(\tilde{w}_i-z_i\tilde{w}_d)+(b+\bz\cdot \ba)\tilde{w}_d \Big|\\
              &\le |a_k|^{-1} \Big(\sum_{i\ne k,d}|a_i||\tilde{w}_i-z_i\tilde{w}_d|+|b+\bz\cdot \ba||\tilde{w}_d|\Big)\\
              &\le |a_k|^{-1} \Big(\sum_{i\ne k,d}|a_i|q^{-\lambda}+|b+\bz\cdot \ba|\xi q^{-\lambda-\mu}\Big)\\
              &\le (d-1)q^{-\lambda}.
  \end{align*}

  \item \textbf{Case $|b+\bz\cdot \ba|\xi q^{-\lambda-\mu}=\max(|a_1|q^{-\lambda},\cdots,|a_{d-1}|q^{-\lambda},|b+\bz\cdot \ba|\xi q^{-\lambda-\mu})$}.\\
  Then we consider the convex body
      \begin{equation*}
      \Sigma_d:=\left\{\bw=(w_1,\ldots,w_d)\in \RR^d: |w_i-z_iw_d|\le 2q^{-\lambda}, i\ne d; \quad |\ba^+\cdot\bw|<1\right\}.
      \end{equation*}
      A direct computation shows that
      \begin{equation*}
      2^{-d}\Vol(\Sigma_d)=|b+\bz\cdot \ba|^{-1}\prod_{i\ne d}2q^{-\lambda}\ge 2^{d-1}(q^\mu+1)^{-1}q^\mu q^{-1}\ge q^{-1}.
      \end{equation*}
      Thus there is a non-zero $\Lambda_P$-lattice point $\tilde{\bw}=(\tilde{w}_1,\ldots,\tilde{w}_d)$ in $\Sigma_k$. Similarly we have
      \begin{align*}
         |\tilde{w}_d|&=|b+\bz\cdot \ba|^{-1}\Big|\sum_{i\ne d}a_i(\tilde{w}_i-z_i\tilde{w}_d) \Big|\\
              &\le |b+\bz\cdot \ba|^{-1} \Big(\sum_{i\ne d}|a_i||\tilde{w}_i-z_i\tilde{w}_d|\Big)\\
              &\le |b+\bz\cdot \ba|^{-1} \Big(\sum_{i\ne d}2|a_i|q^{-\lambda}\Big)\\
              &\le 2(d-1)\xi q^{-\lambda-\mu}.
  \end{align*}
\end{enumerate}
In each case above we set $\bv^+(B,P)=\bw$ and this completes the proof .
\end{proof}

\begin{remark}
Let $\Pi_{B,P}$ denote the subset of $\RR^d$ defined by the inequalities given in \eqref{def-parallel}. Note that the volume of $\Pi_{B,P}$ may be smaller than $1/q$, so the above lemma does not follow directly from Minkowski's linear forms Theorem.
\end{remark}

\section{Some subdivisions}\label{dec-q d}
As aforementioned, we will use the hyperplane potential game in establishing Theorem \ref{main3}. This section is devoted to some preparations for playing hyperplane potential game on $U$ defined in \eqref{def-u}. Hence we will fix $\beta\in(0,1)$ and $\gamma>0$, and a closed ball $B_0\in\sB$ in this section. We are going to define subfamilies $\sB_n(n\ge 0)$ of $\sB$ and decompositions of $\QQ^d$ with respect to the $\beta,\gamma$ and $B_0$ given.

Firstly, denote
\begin{equation*}
\kappa:= \max_{(\bx,y,\bz)\in B_0}\max\{\|\bx\|_{\infty},|y|,\|\bz\|_{\infty}\}+1
\end{equation*}
Then choose a positive number $R$ satisfying
\begin{equation}\label{def-r}
R\ge\max\{4\beta^{-1},10^4d^6\kappa^4\},
\text{\ \ and\ } (R^\gamma-1)^{-1}\leq\left(\frac{\beta^2}{3}\right)^\gamma,
\end{equation}
and set
\begin{equation}\label{values}
 \epsilon=10^{-2}d^{-6}\kappa^{-2}R^{-20d^2}\rho_0.
\end{equation}

Let $\sB_0=\{B_0\}$.
For $n\ge1$, let $\sB_n$ be the subfamily of $\sB$ defined by
\begin{equation*}
\sB_n:=\{B\subset B_0:\beta R^{-n}\rho_0<\rho(B)\le R^{-n}\rho_0\}.
\end{equation*}
In view of \eqref{def-r}, the families $\sB_n$ are mutually disjoint.

Let $n\ge0$, and fix a closed ball $B\in\sB_n$ in this paragraph. We define
$$\cV_B:=\left\{ P \in \QQ^d: H_n\le H_B(P)\le 2H_{n+1}\right\}$$
where
$$H_n=2d^2\epsilon \kappa \rho_0^{-1}R^{n+1}.$$
It follows from \eqref{E:qq} that if $P\in\cV_B$, then
\begin{equation*}
H_n^{\frac{1}{1+\lambda}}\le q(P)\le 2H_{n+1}.
\end{equation*}
We shall also need the following subdivisions of $\cV_{B}$:
$$\cV_{B,1}:=\left\{P\in\cV_B: H_n^{\frac{1}{1+\lambda}}\le q(P)\le H_n^{\frac{1}{1+\lambda}}R^{10d^2}\right\}$$

$$\cV_{B,k}:=\left\{P\in\cV_B: H_n^{\frac{1}{1+\lambda}}R^{10d^2+(2k-4)d}\le q(P)\le H_n^{\frac{1}{1+\lambda}}R^{10d^2+(2k-2)d}\right\}, \quad k\ge2.$$
One can show an important inequality here: for $P\in \cV_{B,k}, k\ge2, $
\begin{equation}\label{ine qxi}
\frac{\xi(B,P)}{q(P)^\lambda}=\frac{H_B(P)}{q(P)^{1+\lambda}}\le \frac{2H_{n+1}}{H_nR^{(1+\lambda)(10d^2+(2k-4)d)}}\le 2R^{-8d^2-2kd+1}.
\end{equation}

Now we define a subfamily $\sB_n'$ of $\sB_n$ inductively as follows. Let $\sB_0'=\{B_0\}$. If $n\ge 1$ and $\sB_{n-1}'$ has been defined, we let
\begin{equation*}
\sB_n':=\left\{B\in \sB_n: B\subset B' \text{ for some } B'\in \sB_{n-1}', \text{ and } B\cap \bigcup_{P\in \cV_B}\Delta_{\epsilon}(P)=\emptyset\right\}
\end{equation*}

The following lemma plays an important role in the proof of Theorem \ref{main3}.
\begin{lemma}\label{impor-lemma}
Let $n\ge 0, B\in \sB_n'$. Then for any $P\in \QQ^d$ with $q(P)^{1+\lambda}\le 2H_{n+1}$, we have $\Delta_{\epsilon}(P)\cap B=\emptyset$.
\end{lemma}

\begin{proof}
Note that $2H_1<1$, hence we may assume that $n\ge 1$. We denote $B_n=B$, and let $B_n\subset\cdots \subset B_0$ be such that $B_k\in \sB_k'$. Assume the contrary that the conclusion of the lemma is not true. Then there exists $P=(\frac{\bp}{q},\frac{r}{q})\in \QQ^d$ with $q^{1+\lambda}\le 2H_{n+1}$ such that $\Delta_{\epsilon}(P) \cap B_k \ne \emptyset $ for every $1\le k\le n$. It then follows from the definition of $\sB_k'$ that $P\notin \cV_{B_k}$, that is,
\begin{equation}\label{ine-1}
H_{B_k}(P)\notin [H_k,2H_{k+1}], \ \forall 1\le k\le n.
\end{equation}
Let $1\le n_0\le n$ be such that
\begin{equation}\label{ine-2}
2H_{n_0}< q^{1+\lambda}\le 2H_{n_0+1}.
\end{equation}
We claim that
\begin{equation}\label{ine-3}
H_{B_k}(P)<H_k, \ \forall 1\le k\le n_0.
\end{equation}
We prove the above claim inductively as follows. Since $H_{B_{n_0}}(P)\le q^{1+\lambda}\le 2H_{n_0+1}$, it follows from \eqref{ine-1} that \eqref{ine-3} holds for $k=n_0$. Suppose that $1\le k\le n_0-1$ and \eqref{ine-3} holds if $k$ is replaced by $k+1$. We prove that
\begin{equation}\label{ine-4}
H_{B_k}(P)\le 2 H_{B_{k+1}}(P).
\end{equation}
Denote $\ba^+(B_i,P)=(\ba_i,b_i), \bz_{B_i}=\bz_i (i=k,k+1)$.
We claim that
\begin{equation}\label{incl-k}
\ba^+(B_{k+1},P)\in \sA(B_k,P).
\end{equation}
Since $\ba^+(B_{k+1},P)\in \sA(B_{k+1},P)$, it is clear that
\begin{equation*}
(\ba_{k+1},b_{k+1})\ne (\bzero,0), \ \ba_{k+1}\cdot \bp+ b_{k+1}r + c_{k+1}q=0, \ \|\ba_{k+1}\|_{\infty}\le q^\lambda.
\end{equation*}
On the other hand, it follows from \eqref{ine-2} and the induction hypothesis that
\begin{align}
|b_{k+1}+\bz_k\cdot \ba_{k+1}| &\le |b_{k+1}+\bz_{k+1}\cdot \ba_{k+1}|+|(\bz_{k}-\bz_{k+1})\cdot \ba_{k+1}|\notag \\
                               &\le |b_{k+1}+\bz_{k+1}\cdot \ba_{k+1}|+d\|\ba_{k+1}\|_{\infty}\rho(B_k)\notag\\
                               &\le q^\mu+\rho(B_{k+1})^{\frac{1}{2}}+dq^{-1}H_{B_{k+1}}(P)\rho(B_k)\notag \\
                               &\le q^\mu+ (\beta R)^{-\frac{1}{2}}\rho(B_k)^{\frac{1}{2}}+dH_{n_0}^{-\frac{1}{2}}H_{k+1}\rho(B_k)\notag \\
                               &\le q^\mu+\frac{1}{2} \rho(B_k)^{\frac{1}{2}}+ d(2d^2\epsilon\kappa R)^{\frac{1}{2}}\rho(B_k)^{\frac{1}{2}}\notag\\
                               &\le q^\mu+\rho(B_k)^{\frac{1}{2}}.\notag
\end{align}
This proves our claim \eqref{incl-k}. It then follows from \eqref{incl-k} and \eqref{E:max} that
\begin{align*}
H_{B_k}(P) &=q \max\{\|\ba_k\|_{\infty},|b_k+\bz_k\cdot \ba_k|\} \\
           &\le q\max\{\|\ba_{k+1}\|_{\infty}, |b_{k+1}+\bz_{k}\cdot \ba_{k+1}|\} \\
           &\le q\max\{\|\ba_{k+1}\|_{\infty}, |b_{k+1}+\bz_{k+1}\cdot \ba_{k+1}|+d\|\ba_{k+1}\|_{\infty}\rho(B_k)\}\\
           &\le 2q\max\{\|\ba_{k+1}\|_{\infty}, |b_{k+1}+\bz_{k+1}\cdot \ba_{k+1}|\} \\
           &=2H_{B_{k+1}}(P).
\end{align*}
Thus \eqref{ine-4} holds. It follows from \eqref{ine-4} and the induction hypothesis that $H_{B_k}(P)\le 2H_{k+1}$. By \eqref{ine-1}, we have $H_{B_k}(P)<H_k$. Thus the claim \eqref{ine-3} follows. This means that $H_{B_1}(P)<H_1<1$, a contradiction. This completes the proof.
\end{proof}

\section{Proof of Theorem \ref{main3}}
At first, we prove the following proposition which plays a key role in the proof of Theorem \ref{main3}.
\begin{proposition}\label{main prop}
Fix $\beta\in(0,1)$ and $\gamma>0$, and a closed ball $B_0\in\sB$ as in Subsection \ref{dec-q d}. Let $R$ be a positive number satisfying \eqref{def-r} and $\epsilon$ given by \eqref{values}. For $n\ge0$, $B\in\sB_n'$ and $k\ge 1$, consider the set
\begin{equation*}
\sC_{B,k,\epsilon}=\Big\{(B',P)\in \sB\times \QQ^d: B'\in \sB_{n+k}, \ B'\subset B, \ P\in \cV_{B',k}, \text{\ and\ } \Delta_{\epsilon}(P) \cap B \neq \emptyset \Big\}.
\end{equation*}
Then there exists an affine hyperplane $E_k(B)\subset \RR^{2d-1}$ such that for any $(B',P)\in \sC_{B,k,\epsilon}$, we have
\begin{equation*}
\Delta_{\epsilon}(P)\cap B' \subset E_k(B)^{(R^{-(n+k)}\rho_0)}.
\end{equation*}
\end{proposition}

We shall need the following the two lemmas.
\begin{lemma}\label{main estimate}
 Let $(B_1,P_1),(B_2,P_2)\in \sC_{B,k,\epsilon}$, and $F_{B_2,P_2}$ be the function defined in \eqref{def-f-bp}, then one has
\begin{equation*}
\left|F_{B_2,P_2}(P_1)\right|\le 30d^4\kappa^2\epsilon q_1^{-1}R^{e_k+k+1}
\end{equation*}
with
\begin{equation*}
e_k=\begin{cases}
10d^2, & k=1 ;\\
2d, &k> 1.
\end{cases}
\end{equation*}

\end{lemma}

\begin{proof}
 Write $P_j=(\frac{\bp_j}{q_j},\frac{r_j}{q_j})$ and let $(\bx_j,y_j,\bz_j)\in \Delta_{\epsilon}(P_j)\cap B, j=1,2.$ Then
\begin{equation*}
\left\|\bx_j-\frac{\bp_j}{q_j}-\left(y_j-\frac{r_j}{q_j}\right)\bz_j\right\|_{\infty}<\frac{\epsilon}{q_j^{1+\lambda}}, \quad \left|y_j-\frac{r_j}{q_j}\right|<\frac{\epsilon}{q_j^{1+\mu}}.
\end{equation*}
The latter inequality implies that
\begin{equation*}
\left|\frac{r_j}{q_j}\right|\le |y_j|+\frac{\epsilon}{q_j^{1+\mu}}\le \kappa.
\end{equation*}
One has
\begin{align*}
&\quad \ \left\|\frac{\bp_1}{q_1} -\frac{\bp_2}{q_2}-\left(\frac{r_1}{q_1}-\frac{r_2}{q_2}\right)\bz_{B_2}\right\|_{\infty} \\
 &=\Big\|-\left(\bx_1-\frac{\bp_1}{q_1}-\left(y_1-\frac{r_1}{q_1}\right)\bz_1\right)+ \left(\bx_2-\frac{\bp_2}{q_2}-\left(y_2-\frac{r_2}{q_2}\right)\bz_2\right) +(\bx_1-\bx_2) \\
 &\ \ \  + \frac{r_1}{q_1}(\bz_1-\bz_{B_2})+ \frac{r_2}{q_2}(\bz_{B_2}-\bz_2)+ (y_1\bz_1-y_2\bz_2)\Big\| \\
 &\le \left\|\bx_1-\frac{\bp_1}{q_1}-\left(y_1-\frac{r_1}{q_1}\right)\bz_1\right\|_{\infty}
 + \left\|\bx_2-\frac{\bp_2}{q_2}-\left(y_2-\frac{r_2}{q_2}\right)\bz_2\right\|_{\infty} \\
 &\ \ \ + \left\|\bx_1-\bx_2\right\|_{\infty}+ \left\|\frac{r_1}{q_1}(\bz_1-\bz_{B_2})\right\|_{\infty}+ \left\|\frac{r_2}{q_2}(\bz_{B_2}-\bz_2)\right\|_{\infty}+ \left\|y_1\bz_1-y_2\bz_2\right\|_{\infty}\\
 &\le \frac{\epsilon}{q_1^{1+\lambda}}+\frac{\epsilon}{q_2^{1+\lambda}}+10\kappa\rho(B)
\end{align*}
and
\begin{equation*}
\left|\frac{r_1}{q_1}-\frac{r_2}{q_2}\right|=\left|-\left(y_1-\frac{r_1}{q_1}\right)+\left(y_2-\frac{r_2}{q_2}\right)+(y_1-y_2)\right|\le \frac{\epsilon}{q_1^{1+\mu}}+\frac{\epsilon}{q_2^{1+\mu}}+2\rho(B).
\end{equation*}
Then it follows that
\begin{align*}
\left|F_{B_2,P_2}(P_1)\right| &=\left|\ba_2\cdot\left(\frac{\bp_1}{q_1}-\frac{\bp_2}{q_2}\right)+ b_2\left(\frac{r_1}{q_1}-\frac{r_2}{q_2}\right)\right| \\
                          &=\left|\ba_2\cdot\left(\frac{\bp_1}{q_1}-\frac{\bp_2}{q_2}-\left(\frac{r_1}{q_1}-\frac{r_2}{q_2}\right)\bz_{B_2}\right)+
                           (b_2+\bz_{B_2}\cdot\ba_2)\left(\frac{r_1}{q_1}-\frac{r_2}{q_2}\right)\right|\\
                          &\le d\|\ba_2\|_\infty\left(\frac{\epsilon}{q_1^{1+\lambda}}+\frac{\epsilon}{q_2^{1+\lambda}}+10\kappa\rho(B)\right)+ |b_2+\bz_{B_2}\cdot\ba_2|\left(\frac{\epsilon}{q_1^{1+\mu}}+\frac{\epsilon}{q_2^{1+\mu}}+2\rho(B)\right)\\
                          &\le dq_2^\lambda\left(\frac{\epsilon}{q_1^{1+\lambda}}+\frac{\epsilon}{q_2^{1+\lambda}}\right)
                          +2q_2^\mu\left(\frac{\epsilon}{q_1^{1+\mu}}+\frac{\epsilon}{q_2^{1+\mu}}\right)+12d\kappa\rho(B)\max\{ \|\ba_2\|_\infty,|b_2+\bz_{B_2}\cdot\ba_2|\} \\
                          &\le d\epsilon q_1^{-1}\left(\frac{q_2^{\lambda}}{q_1^{\lambda}}+\frac{q_1}{q_2}+2\frac{q_2^{\mu}}{q_1^{\mu}}+2\frac{q_1}{q_2}\right)
                          +12d\kappa R^{-n}\rho_0q_2^{-1}H_{B_2}(P_2)\\
                          &\le 6d\epsilon q_1^{-1}R^{e_k}+48d^3\epsilon \kappa^{2}q_2^{-1}R^{k+1} \\
                          &\le 30d^4\kappa^2\epsilon q_1^{-1}R^{e_k+k+1}.
\end{align*}
\end{proof}

\begin{lemma}\label{point-hyperplane}
For any $(B_1,P_1),(B_2,P_2)\in \sC_{B,k,\epsilon}$, we have $F_{B_2,P_2}(P_1)=0$.
\end{lemma}

\begin{proof}
For simplicity, we write the objects $\ba^+(B_j,P_j)$, $\bv^+(B_j,P_j)$, $\xi_{B_j,P_j}$, $F_{B_j,P_j}$, $\cL_{B_j,P_j}$, $\cH_{B_j,P_j}$ $(j=1,2)$ as $\ba^+_j$, $\bv^+_j$, $\xi_j$, $F_j$, $\cL_j$, $\cH_j$ respectively. We are divided into three cases:
\begin{enumerate}
\item \textbf{Case $k=1$}.
Then by Lemma \ref{main estimate}, we have
\begin{equation*}
q_1|F_{2}(P_1)|\le 30d^4\kappa^2\epsilon R^{e_k+k+1}<1.
\end{equation*}
As $q_1|F_{2}(P_1)|\in \ZZ$, we have $F_{2}(P_1)=0$.

\item \textbf{Case $k\ge 2$ and $\cL_{1}$ parallel to $\cH_{2}$, that is
\begin{equation}\label{parallel-av}
\ba_{2}^+\cdot \bv_{1}^+=0.
\end{equation}}
Assume the contrary that $F_{B_2,P_2}(P_1)\ne0$. Write $\bv_{1}^+=(\bv_1,u_1)=(v_{1,1},\ldots,v_{d-1,1},u_1)$. We claim that
\begin{equation}\label{claim1}
q_1\big|F_{2}(P_1)v_{i,1}\big|,\  q_1\big|F_{2}(P_1)u_{1}\big|\in \ZZ \  \text{for each } 1\le i \le d-1.
\end{equation}
Indeed, since $\bv_{1}^+\in \Lambda_{P_1}\setminus \{\bzero\}$, we can write
\begin{equation}\label{value-v-p1}
\bv_{1}^+=c \left(\frac{\bp_1}{q_1},\frac{r_1}{q_1}\right)+\bc,
\end{equation} where  $c\in\ZZ$, $\bc\in\ZZ^d$.
Combining \eqref{parallel-av} and \eqref{value-v-p1}, we get
\begin{equation}\label{pr-claim1}
c F_{2}(P_1)\in \ZZ.
\end{equation}
According to \eqref{value-v-p1}, $q_1v_{i,1}, q_1u_1 \in c\ZZ+q_1\ZZ$. Then the claim
\eqref{claim1} follows directly from \eqref{pr-claim1}.

Note that $\bv^+_{1}\ne \bzero$. It follows from \eqref{claim1} that
\begin{equation*}
q_1\Big|F_{2}(P_1)\Big|\Big(\sum_{1\le i\le d-1}|v_{i,1}|+|u_1|\Big)\ge 1.
\end{equation*}
However, according to Lemma \ref{main estimate} and \eqref{ine qxi}, we have
\begin{eqnarray*}
&&\quad q_1\Big|F_{2}(P_1)\Big|\Big(\sum_{1\le i\le d-1}|v_{i,1}|+|u_1|\Big)\\
&&\le 30d^4\kappa^2\epsilon R^{e_k+k+1} \Big(2d(d-1)q_1^{-\lambda}+2d\xi_1 q_1^{-\lambda-\mu}\Big)\\
&&\le 60d^6 \kappa^2\epsilon R^{e_k+k+1} \xi_1q_1^{-\lambda}\\
&&\le 120d^6 \kappa^2\epsilon R^{2d+k+1} R^{-8d^2-2kd+1} \\
&&<1
\end{eqnarray*}
which leads to a contradiction.

\item \textbf{Case $k\ge 2$ and $\cL_{1}$ intersects $\cH_{2}$}.\\
Assume the contrary that $F_{B_2,P_2}(P_1)\ne0$. Let
$$P_0=\frac{\bp_0^+}{q_0}=(\frac{\bp_0}{q_0},\frac{r_0}{q_0})$$
be their intersection. Write
\begin{equation*}
\frac{\bp_0^+}{q_0}=\frac{\bp_1^+}{q_1}+t_0\bv_{1}^+.
\end{equation*}
Then $\left(\frac{\bp_0^+}{q_0},t_0\right)^T$ is the solution of the following linear equations
\begin{equation*}
 \begin{pmatrix}q_1I_d & -q_1\bv_1^{+T} \\ \ba_2^+ & 0 \end{pmatrix} \begin{pmatrix}\bw \\ t
  \end{pmatrix}=
 \begin{pmatrix}\bp_1^{+T} \\ C_2 \end{pmatrix}
\end{equation*}
where $\bv_1^{+T}, \bp_1^{+T}$ means the transport of $\bv_1^+, \bp_1^+$, and $C_2=C(B_2,P_2)$ is defined in \eqref{def-c-bp}. Let $M$ be the following matrix
\begin{equation*}
 \begin{pmatrix}q_1I_d & -q_1\bv_1^{+T}  &\bp_1^{+T}\\ \ba_2^+ & 0 & C_2\end{pmatrix}
\end{equation*}
and $M_i (1\le i\le d+2)$ be the matrix obtained by deleting the $i$-th column of $M$. In view of the fact that $\bv_1^+\in \Lambda_{P_1}$, a simple computation immediately implies
\begin{equation}\label{compu-det}
\det(M_i)\in q_1^{d-1}\ZZ \quad(1\le i\le d+2).
\end{equation}
By Cramer's rule,
\begin{equation}\label{cramer}
\left(\frac{\bp_0^+}{q_0},t_0\right)=\left(\frac{\det(M_1)}{\det(M_{d+2})},\cdots, \frac{\det(M_{d+1})}{\det(M_{d+2})}\right).
\end{equation}
Hence
\begin{equation}\label{lambda0}
|t_0|=\left|\frac{\det(M_{d+1})}{\det(M_{d+2})}\right|=\frac{|F_2(P_1)|}{|\ba_{2}^+\cdot \bv_{1}^+|}.
\end{equation}
In view of \eqref{compu-det} and \eqref{cramer}, we have
\begin{equation}\label{ine-q0}
q_0\le q_1^{-d+1}|\det(M_{d+2})|=q_1|\ba_{2}^+\cdot \bv_{1}^+|.
\end{equation}
It is clear that
\begin{align}\label{ine-av}
|\ba_{2}^+\cdot \bv_{1}^+|
&\le \sum_{1\le i\le d-1}\left|\xi_2\cdot 2dq_1^{-\lambda}\right|+\left|q_2^{\mu}\cdot2d\xi_1 q_1^{-\lambda-\mu}\right| \notag\\
&\le 2d(d-1)R^{e_k}\xi_1q_1^{-\lambda}+2dR^{e_k}\xi_1 q_1^{-\lambda} \notag\\
&\le 2d^2R^{e_k}\xi_1 q_1^{-\lambda}\notag\\
&\le 2d^2R^{-8d^2-(2k-2)d+1}.
\end{align}
It follows that
\begin{align}\label{ine-av1}
\frac{q_0}{q_1}&\le |\ba_{2}^+\cdot \bv_{1}^+| \notag\\
&\le 2d^2R^{-8d^2-(2k-2)d+1}\notag\\
&\le \frac{1}{2}.
\end{align}
Combine the inequalities \eqref{ine-q0}, \eqref{ine-av}, \eqref{ine qxi} and the obvious estimate $\lambda \ge 1/d$, we have
\begin{align}\label{ine-q0-2}
q_0^{1+\lambda} &\le q_1^{1+\lambda}|\ba_{2}^+\cdot \bv_{1}^+|^{1+\lambda}  \notag\\
          &\le q_1^{1+\lambda}(2d^2R^{e_k}\xi_1 q_1^{-\lambda})^{1+\lambda} \notag\\
          &\le 4d^4R^{4d}(\xi_1^\lambda q_1^{-\lambda^2})H_{B_1}(P_1) \notag\\
          & < 4d^4R^{4d}R^{-8d-2k+1}2R^k H_n \notag\\
          &\le 8d^4R^{-4d-k+1}H_n\notag\\
          &\le H_n.
\end{align}
Note that
\begin{equation*}
\left\|\frac{\bp_1}{q_1}-\frac{\bp_0}{q_0}-\left(\frac{r_1}{q_1}-\frac{r_0}{q_0}\right)\bz_{B_1}\right\|_{\infty}=|t_0|\left\|\bv_1-\bz_{B_1}u_1\right\|_{\infty}\le \frac{2d|t_0|}{q_1^{\lambda}}
\end{equation*}
and
\begin{equation*}
\left|\frac{r_1}{q_1}-\frac{r_0}{q_0}\right|= |t_0 u_1|\le \frac{2d|t_0|\xi_1}{q_1^{\lambda+\mu}}.
\end{equation*}

We claim that
\begin{equation}\label{contain}
\Delta_{\epsilon}(P_1)\cap B\subset \Delta_{\epsilon}(P_0).
\end{equation}
In view of Lemma \ref{impor-lemma} and \eqref{ine-q0-2},  \eqref{contain} will contradict to the assumption that $B\in \sB_n'$. It remains to prove \eqref{contain}. Indeed, for $(\bx,y,\bz)\in\Delta_{\epsilon}(P_1)\cap B$, by \eqref{ine-av1}, \eqref{lambda0} and \eqref{ine-q0} we have
 \begin{align*}
 q_0^{1+\mu}\left|y-\frac{r_0}{q_0}\right|
&\le q_0^{1+\mu}\left|y-\frac{r_1}{q_1}\right|+q_0^{1+\mu}\left|\frac{r_1}{q_1}-\frac{r_0}{q_0}\right| \\
&\le q_0^{1+\mu}\frac{\epsilon}{q_1^{1+\mu}}+q_0^{1+\mu}\frac{2d|t_0|\xi_1}{q_1^{\lambda+\mu}} \\
&\le \frac{\epsilon}{2}+2dq_1|F_2(P_1)|\frac{\xi_1}{q_1^\lambda} \\
&\le \frac{\epsilon}{2}+60d^5\kappa^2R^{2d+k+2-8d^2-2kd}\epsilon \\
&\le \epsilon
\end{align*}
and
\begin{align*}
&\ q_0^{1+\lambda}\left\|\bx-\frac{\bp_0}{q_0}-\left(y-\frac{r_0}{q_0}\right)\bz\right\|_{\infty}\\
&\le q_0^{1+\lambda}\left\|\bx-\frac{\bp_1}{q_1}-\left(y-\frac{r_1}{q_1}\right)\bz\right\|_{\infty}+ q_0^{1+\lambda}\left\|\frac{\bp_1}{q_1}-\frac{\bp_0}{q_0}-\left(\frac{r_1}{q_1}-\frac{r_0}{q_0}\right)\bz\right\|_{\infty}\\
&\le q_0^{1+\lambda}\frac{\epsilon}{q_1^{1+\lambda}}+q_0^{1+\lambda}\frac{2d|t_0|}{q_1^{\lambda}}+ q_0^{1+\lambda}d\left|\frac{r_1}{q_1}-\frac{r_0}{q_0}\right|\|\bz-\bz_{B_1}\|_{\infty} \\
&\le \frac{\epsilon}{2}+2d|q_1F_2(P_1)|\cdot\frac{q_0^\lambda}{q_1^\lambda}+ 2d^2|q_1F_2(P_1)|\cdot\frac{q_0^\lambda}{q_1^\lambda}\cdot2R^{-n}\rho_0\frac{q_1\xi_1}{q_1^{1+\mu}}\\
&\le \frac{\epsilon}{2}+120d^5\kappa^2\epsilon R^{2d+k+1}\cdot2d^2R^{-8d-(2k-2)+1}R^k\\
&\le \epsilon.
\end{align*}
\end{enumerate}
\end{proof}

\begin{proof}[Proof of Proposition \ref{main prop}]
Choose $(B_0',P_0)\in\sC_{B,k,\epsilon}$ such that
$$q_0=q(P_0)=\min\left\{q(P): \exists\text{ closed ball $B'$ with }  (B',P)\in\sC_{B,k,\epsilon} \right\}.$$
Consider the attached hyperplane in $\RR^{2d-1}$
 $$\cH_{B_0',P_0}=\left\{(\bx,y,\bz)\in\RR^{2d-1}: \ba_0\cdot\bx +b_0y-C=0\right\}$$
where $\ba_0^+=(\ba_0,b_0)$ and $C=C(B_0',P_0)$ are given in Subsection \ref{subsc-hyperplane}. We claim that $\cH_{B_0',P_0}$ is the $E_k(B)$ that we need. In other words, for any $(B',P)\in \sC_{B,k,\epsilon}$,
$$\Delta_{\epsilon}(P)\cap B'\subset \cH_{B_0',P_0}^{(R^{-(n+k)}\rho_0)}.$$
Indeed, we have proved in Lemma \ref{point-hyperplane} that $P\in \cH_{B_0',P_0}$ for $(B',P)\in \sC_{B,k,\epsilon}$. Hence for any $(\bx,y,\bz)\in\Delta_{\epsilon}(P)\cap B'$, we have
\begin{align*}
|\ba_0\cdot\bx+b_0y-C| &=\left|\ba_0\cdot\left(\bx-\frac{\bp}{q}\right)+b_0\left(y-\frac{r}{q}\right)\right| \\
                     &\le (d-1)\|\ba_0\|_{\infty}\left\|\bx-\frac{\bp}{q}-\left(y-\frac{r}{q}\right)\bz\right\|_{\infty}+|b_0+\bz\cdot\ba_0|\left|y-\frac{r}{q}\right| \\
&\le (d-1)q_0^{\lambda}\frac{\epsilon}{q^{1+\lambda}}+2q_0^{\mu}\frac{\epsilon}{q^{1+\mu}} \\
&\le (d+1)\frac{\epsilon}{q_0}.
\end{align*}
Denote the width of this thicken hyperplane as $\omega$,  then
\begin{align*}
\omega &\le \frac{(d+1)\epsilon}{q_0 \max\{\|\ba_0\|_{\infty},|b_0|\}} \\
       &\le \frac{(d+1)\epsilon}{(1+(d-1)\kappa)^{-1}q_0\max\{\|\ba_0\|_{\infty},|b_0+\bz\cdot\ba_0|\}} \\
       &\le \frac{(d+1)(1+(d-1)\kappa)\epsilon}{H_{n+k}}\\
       &\le R^{-(n+k)}\rho_0
\end{align*}
which finishes the proof.
\end{proof}
\begin{proof}[Proof of Theorem \ref{main3}]
In view of Remark \ref{diffeo}, Lemma \ref{HPW} and Lemma \ref{Dani-Klein}, to prove Theorem \ref{main3}, it suffices to show that the set $S(\rr)$ is $(\beta,\gamma)$-HPW for any $\beta\in(0,1)$, $\gamma>0$. Fix $\beta\in(0,1)$ and $\gamma>0$ from now on. Bob starts the $(\beta,\gamma)$-hyperplane potential game on $\RR^{2d-1}$ with target set $S(\rr)$ by choosing a closed ball $B_0\subset\RR^{2d-1}$ of radius $\rho_0$. As discussed in \cite[Remark 2.4]{AGK}, without loss of generality we may assume that Bob will play so that $\rho_0\le1/d$ and $\rho_i:=\rho(B_i)\to0$, where $B_i$ is the ball chosen by Bob at the $i$-th turn. Now we have fixed $\beta\in(0,1)$ and $\gamma>0$, and the closed ball $B_0\in\sB$ as in Proposition \ref{main prop}. Let $R$ be a positive number satisfying \eqref{def-r} and $\epsilon$ be the constant given by \eqref{values}.
Write $i_n$ to be the smallest nonnegative integer with $B_{i_n}\in\sB_n$. Let $\cN$ denote the set of all $n\in\NN$ with $B_{i_n}\in\sB_n'$.

Let Alice play according to the strategy as follows.  At the $i$-th stage, if $i=i_n$ for some $n\in \cN$, then  Alice chooses the family of hyperplane neighborhoods $\{E_k(B_{i_n})^{(3R^{-(n+k)}\rho_0)}:k\in\NN\}$, where the hyperplane $E_k(B_{i_n})$ is given by Proposition \ref{main prop}. Otherwise, Alice makes an empty move.
Since $B_{i_n}\in\sB_n$, it follows that $\rho_{i_n}>\beta R^{-n}\rho_0$. Hence  Alice's move is legal as we have
$$\sum_{k=1}^\infty(3R^{-(n+k)}\rho_0)^\gamma=(3R^{-n}\rho_0)^\gamma(R^\gamma-1)^{-1}\le(\beta \rho_{i_n})^\gamma.$$
We claim that this is a winning strategy for Alice, that is, the point $\bx_\infty=\bigcap_{i=0}^\infty B_i$ lies in the set
$$S(\rr)\cup\bigcup_{n\in\cN}\bigcup_{k=1}^\infty E_k(B_{i_n})^{(3R^{-(n+k)}\rho_0)}.$$

To prove our claim, we are divided into two different cases:
\begin{enumerate}
\item \textbf{Case $\cN=\NN\cup \{0\}$.} For any $P\in \QQ^d$, there is $n$ such that $q^{1+\lambda}\le 2H_{n+1}$. Since $n\in \cN$, we have $B_{i_n}\in \sB_n'$. Then we have $\Delta_{\epsilon}(P)\cap B_{i_n}=\emptyset$ by Lemma \ref{impor-lemma}. Thus it follows from the definition of $S(\rr)$ that $\bx_{\infty}\in S_{\epsilon}(\rr)\subset S(\rr)$. Hence Alice wins.
\item \textbf{Case $\cN\ne \NN\cup \{0\}$.}  Let $n$ be the smallest integer with $n\notin \cN$.  Then we have $B_{i_n}\notin \sB_n'$ and $B_{i_{n-1}}\in \sB_{n-1}'$ as $n-1\in \cN$. By the definition of $\sB_n'$, there exists $P\in \cV_{B_{i_n},k}$ with $1\le k \le n$ and $\Delta_{\epsilon}(P)\cap B_{i_n}\ne \emptyset$. By Proposition \ref{main prop}, we have $\Delta_{\epsilon}(P)\cap B_{i_n}\subset E_k(B_{i_{n-k}})^{(R^{-n}\rho_0)}$. In view of $\rho_{i_n}\le R^{-n}\rho_0$, it follows that $\bx_{\infty}\in B_{i_n} \subset E_k(B_{i_{n-k}})^{(3R^{-n}\rho_0)}$. Hence Alice wins.
\end{enumerate}
This completes the proof of Theorem \ref{main3}.
\end{proof}

\section{Proof of main theorems}
In this section, we deduce Theorem \ref{main1} and Theorem \ref{main2} from Theorem \ref{main3}. Indeed, the argument presented here is similar to the argument presented in \cite[Section 6]{AGK}. For sake of completeness, we reproduce the proof in our setting here.

\begin{proof}[Proof of Theorem \ref{main1}]
The proof is divided into three steps:

\textbf{Step 1.} We show that it suffices to prove the set $E(F^+)$ is HAW on $\ggm$. Indeed, by applying the following diffeomorphism
\begin{equation*}
\tau: \ggm \to \ggm, \tau(g\Gamma)=(g^T)^{-1}\Gamma
\end{equation*}
to the set $E(F^+)$, we can see that the set $E(F^-)$ is also HAW if $E(F^+)$ does, where $F^-$ denotes the subsemigroup $\{e\}\cup (F\setminus F^+)$. Hence, in view of the intersection stability of HAW sets,  $E(F)$ will be HAW if $E(F^+)$ does.

\textbf{Step 2}. We show that it suffices to prove the theorem for $F^+=F_{\rr}^+$, which was defined in Theorem \ref{main3}. Indeed, by the real Jordan decomposition (cf. \cite[Proposition 4.3.3]{Mor}), for any one-parameter diagonalizable  subsemigroup $F^+$, there are one-parameter subsemigroups $F_i^+=\{g_t^{(i)}:t>0\} (i=1,2)$ such that $F_1^+$ is $\RR$-diagonalizable, $F_2^+$ has compact closure, and $g_t=g_t^{(1)}g_t^{(2)}$ with $g_t^{(1)}$ commuting with $g_t^{(2)}$. It is obvious that $E(F^+)=E(F^+_1)$. Hence we are reduced to consider the case that $F$ satisfying \eqref{e:tame} and $\RR$-diagonalizable, which is equivalent to say that there exists $g'\in G$ and $\rr$ satisfying \eqref{def-r-special} such that $F^+=g'F^+_{\rr}g'^{-1}$. Note that in this case we have $E(F^+)=g'E(F_{\br}^+)$. Hence our statement follows from (3) of Lemma \ref{haw-manifold}.

\textbf{Step 3}. We prove the theorem for $F^+_{\rr}$. In view of Lemma \ref{haw-manifold}, we have to prove that for any $\Lambda\in \ggm$, there is an open neighborhood $\Omega$ of $\Lambda$ in $\ggm$ such that $\Omega\cap E(F_{\br}^+)$ is HAW on $\Omega$. Let
\begin{equation*}
P= \left\{ g\in G: g=\begin{pmatrix} T&\bzero \\ N&T' \end{pmatrix},
T\in \GL_{d-1}(\RR), N\in M_{2\times (d-1)}(\RR), T' \text{\ is lower triangular} \right\}.
\end{equation*}
It's not hard to check that for any $g\in P$, the set $\{g_tgg_t^{-1} :t>0\}$ is bounded in $G$. By the Bruhat decomposition, the set $PU$ is Zariski open in $G$ and the multiplication map $P\times U \to PU$ is a diffeomorphism.

 According to the Borel density theorem, the set $\pi^{-1}(\Lambda)$ is Zariski dense in $G$. Hence, $\pi^{-1}(\Lambda)\cap PU\ne \emptyset$, that is, there exists $p_0\in P$ and $u_0\in U$ such that $\Lambda= p_0u_0\Gamma$.

Let $\Omega_P$ and $\Omega_U$ be open neighborhoods of $p_0$ and $u_0$ in $P$ and $U$ respectively, which are small enough such that the map $\phi: \Omega_P\times \Omega_U\to \ggm, \phi(p,u)=pu\Gamma$ is a diffeomorphism onto an open subset $\Omega$ in $\ggm$. In view of Lemma \ref{haw-manifold}(4), it suffices to prove that the set
\begin{equation}\label{def-set}
\phi^{-1}(E(F_{\br}^+)\cap \Omega)= \{(p,u)\in \Omega_P\times \Omega_U:pu\Gamma\in E(F_{\br}^+) \}
\end{equation}
is HAW on $\Omega_P\times \Omega_U$. By the definition of $P$, we have that $pu\Gamma\in E(F_{\br}^+)$ if and only if $u\Gamma\in E(F_{\br}^+)$. It follows that the set
\eqref{def-set} is equal to
\begin{equation*}
\Omega_P \times \{u\in \Omega_U: u\Gamma\in E(F_{\br}^+)\}.
\end{equation*}
Then it from Theorem \ref{main3} and (5) of Lemma \ref{haw-manifold} that the set $E(F_{\rr}^+)$ is HAW.
\end{proof}

\begin{proof}[Proof of Theorem \ref{main2}]
We will prove the theorem only for $F^+=F^+_{\rr}$ with $\rr$ satisfying \eqref{def-r-special} here, since the proof for general $F^+$ satisfying \eqref{e:tame} follows along the same lines as Step 2 of the proof of Theorem \ref{main1} and will be omitted. There are two subcases.
\begin{enumerate}
\item \textbf{Case $r_1>r_d$.} Then it is easy to check that $H(F^+)$ is equal to $U$. We need to prove that for any $\Lambda\in \ggm$, the set $u\in U$ such that $u\Lambda\in E(F^+)$ is HAW on $U$. In view of Lemma \ref{haw-manifold}, it suffices to prove that for any $u_0\in U$, there is an open neighborhood $\Omega$ of $u_0$ in $U$ such that the set
\begin{equation}\label{def-omega}
\{u\in \Omega: u\Lambda \in E(F^+)\}
\end{equation}
is HAW on $\Omega$. Similar to the proof of Theorem \ref{main1}, the Bruhat decompostion and the Borel density theorem imply that $\pi^{-1}(\Lambda)\cap u_0^{-1}PU\ne \emptyset$. Choose $g_0\in \pi^{-1}(\Lambda)\cap u_0^{-1}PU$. Then $\Lambda=g_0\Gamma$ and $u_0g_0\in PU$. Let $\Omega_1$ be an open neighborhood of $u_0$ in $U$ with $\Omega_1 g_0\subset PU$. Then there are smooth maps $\phi:\Omega_1\rightarrow P$ and $\psi:\Omega_1\rightarrow U$ such that
\begin{equation}\label{phi-psi}
ug_0=\phi(u)\psi(u) \quad \forall u\in \Omega_1
\end{equation}

We claim that
\begin{equation}\label{claim-iso}
\text{the tangent map $(d\psi)_{u_0}$ is a linear isomorphism.}
\end{equation}
The set \eqref{def-omega} is HAW follows from our claim. Indeed, assuming \eqref{claim-iso}, we can find a neighborhood $\Omega\subset \Omega_1$ such that $\psi$ is a diffeomorphism when restricted on $\Omega$. Note that for $u\in \Omega$, the set
$$F^+_{\rr}u\Lambda=F^+_{\rr}u g_0\Gamma=F^+_{\rr}\phi(u)\psi(u)\Lambda$$
is bounded if and only if $=F^+_{\rr}\psi(u)\Lambda$ is bounded. Hence, in view of Theorem \ref{main3} and Lemma \ref{haw-manifold}, we prove that the set \eqref{def-omega} is HAW modulo Claim \eqref{claim-iso}.

Let's turn to the proof of Claim  \eqref{claim-iso}. Write $p_1=\phi(u), u_1=\psi(u)$, then it follows from \eqref{phi-psi} that
\begin{equation}\label{tagent}
dr_{g_0}(Y)=dr_{u_1}\circ (d\phi)_{u_0}(Y)+dl_{p_1}\circ (d\psi)_{u_0}(Y),\quad \forall Y\in T_{u_0}U.
\end{equation}
If $(d\psi)_{u_0}(Y)=0$, then the left-hand side of \eqref{tagent} belongs to $T_{u_0g_0}(Uu_0g_0)$ and right-hand side belongs to $T_{u_0g_0}(Pu_0g_0)$, thus $Y=0$. This proves Claim \eqref{claim-iso}.

\item \textbf{Case $r_1=r_d=\frac{1}{d}$.} In this case, the expanding horospherical subgroup $H$ coincide with the subgroup $U_0$ defined as
\begin{equation*}
U_0:=\{u_{\bx}: \bx\in \RR^d\}, \quad \text{where } u_{\bx}=\begin{pmatrix} Id&\bx \\0&1 \end{pmatrix} \in G.
\end{equation*}
In view of the correspondence presented in \cite[Theorem 2.20]{D1}, the set $\{x\in \RR^d: u_\bx \Gamma \in E(F^+)\}$ coincides with the set of badly approximable vectors $\Bad_d$, which is proved to be HAW already in \cite{BFKRW}. Then we omit the remaining part of the proof here, since it is similar to the proof of the above case $r_1>r_d$.
\end{enumerate}

\end{proof}

\section*{acknowledgements}
The authors are grateful to Jinpeng An for helpful suggestions and comments. The research of the second author is supported by CPSF (\#2015T80010).

\end{document}